\documentclass{amsart}
\usepackage{color}
\usepackage{stmaryrd}
\usepackage{amssymb,latexsym,amsmath, amsfonts}
\usepackage{bbm}

\usepackage{color}
\usepackage{slashed}

\newcommand{\N}{{\mathbb N}}
\newcommand{\Z}{{\mathbb Z}}
\newcommand{\R}{{\mathbb R}}
\newcommand{\C}{{\mathbb C}}

\newcommand{\e}{\mathrm{e}}

\newtheorem{thm}{Theorem}[section]
\newtheorem{lem}[thm]{Lemma}
\newtheorem{cor}[thm]{Corollary}
\newtheorem{prop}[thm]{Proposition}

\theoremstyle{definition}
\newtheorem{rem}[thm]{Remark}
\newtheorem{dfn}[thm]{Definition}

\title[Dixmier traces of Hankel operators]{Estimating Dixmier traces of Hankel operators in Lorentz ideals}
\author{Magnus Goffeng, Alexandr Usachev}
\address{
Email: goffeng@chalmers.se, usachev@chalmers.se\newline
\indent Department of Mathematical Sciences\newline 
\indent Chalmers University of Technology and \newline 
\indent University of Gothenburg\newline 
\indent SE-412 96 Gothenburg\newline 
\indent Sweden\newline}
\date{\today}

\begin{document}

\begin{abstract}
In this paper we study Dixmier traces of powers of Hankel operators in Lorentz ideals. We extend results of Engli\v{s}-Zhang to the case of powers $p\geq 1$ and general Lorentz ideals starting from abstract extrapolation results of Gayral-Sukochev. In the special case $p=2,4,6$ we give an exact formula for the Dixmier trace. For general $p$, we give upper and lower bounds on the Dixmier trace. We also construct, for any $p$ and any Lorentz ideal, examples of non-measurable Hankel operators. 
\end{abstract}

\maketitle

\section{Introduction}

The construction of Dixmier traces goes back to work of Dixmier \cite{Dixie} who was motivated by the problem of finding a non-normal trace on the von Neumann algebra of bounded operators. Since then Dixmier traces have  taken a prominent role in Connes' program for noncommutative geometry \cite{C_book} and found applications in the analysis of rough structures such as Julia sets \cite{juliaconnes}, limit sets of quasi-Fuchsian groups \cite{quasiconnes} and in complex geometry \cite{EZe,GG0,GG}. The non-normality of the Dixmier trace and the non-separability of its domain of definition makes computations and estimates of Dixmier traces a challenging problem. In this paper we propose a methodology to estimate Dixmier traces of powers of Hankel operators, building on work of Gayral-Sukochev \cite{GS}.

The inspiration for this work is a paper by Engli\v{s}-Zhang \cite{EZ} where Dixmier traces of Hankel operators in the Lorentz ideal $\mathcal{M}_{1,\infty}$ were estimated by means of Besov norms. Recent work in fractal geometry \cite{juliaconnes,quasiconnes} and the questions posed in \cite[Section 7.3]{EZ} lead us to ask for an extension of the estimates in \cite{EZ} to powers $p\geq 1$ and more general Lorentz ideals. The approach we take in this paper differs from that of \cite{EZ}. Our method consists of a rather straightforward application of extrapolation results of Gayral-Sukochev \cite{GS}. 

In the classical examples, naturally appearing physical and geometrical operators are measurable, that is all Dixmier traces take the same value on such operators. An example of a non-measurable pseudo-differential operator with symbol of H\"ormander type $(1,0)$ can be found in \cite[Proposition 11.3.22]{LSZ}. Engli\v{s}-Zhang \cite[Theorem 4]{EZ} constructed a non-measurable Hankel operator from $\mathcal{M}_{1,\infty}$. We show that there are non-measurable Hankel operators in any ($p$-convexified) Lorentz ideal.

Let us summarize our main results in a theorem. For a function $f$ on the circle $S^1$, we let $H_{\bar{f}}$ denote the associated Hankel operator on the Hardy space $H^2(S^1)$ (see more below in Section \ref{hanekopdldl} below). We let $\mathcal{M}^{(p)}_\psi$ denote the $p$-convexification of the Lorentz ideal $\mathcal{M}_\psi$ and $\mathcal{M}^{(p)}_{\psi,0}$ its separable subspace (see more in Subsection \ref{opidelalsos} below), and let $\mathrm{Tr}_{\omega,\psi}: \mathcal{M}_\psi\to \C$ be the Dixmier trace associated with an exponentiation invariant extended limit $\omega$. We write $A\sim B$ if there is a universal constant $C>0$ such that $C^{-1}A\leq B\leq CB$. When saying universal, we still allow for a dependence on $p$ and $\psi$.

\begin{thm}
\label{mainthmintro}
Let $p\geq 1$, $(\|\cdot\|_{B^{1/q}_{q,q},*})_{q\geq p}$ a family of norms on the Besov spaces $B^{1/q}_{q,q}(S^1)$ for $q\geq p$ satisfying the conditions of Corollary \ref{equivbesovgeneerl}, and $\psi:[0,\infty)\to [0,\infty)$ be an increasing concave function with regular variation of index $0$ satisfying $\psi(0)=0$, $\lim_{t\to \infty}\psi(t)=\infty$ and the conditions \eqref{apsicond} and  \eqref{Phi_as1}. Then for any holomorphic function $f$ the following holds:
\begin{enumerate}
\item[I)] $H_{\bar{f}}\in \mathcal{M}_{\psi}^{(p)}$ if and only if $\sup_{q>p}\frac{1}{\psi(\mathrm{e}^{(q-p)^{-1}})} \|f\|_{B^{1/q}_{q,q},*}^q<\infty$.
\item[II)] For any exponentiation invariant extended limit $\omega$, 
$$\mathrm{Tr}_{\omega,\psi}(|H_{\bar{f}}|^p)\sim\lim_{q-p\to  \tilde{\omega}}\frac{1}{\psi(\mathrm{e}^{(q-p)^{-1}})} \|f\|_{B^{1/q}_{q,q},*}^q.$$
Here $\tilde{\omega}$ is defined as in Equation \eqref{omegatildedef} (see page \pageref{omegatildedef}). 
\item[III)] It holds that
$$\mathrm{d}_{\mathcal{M}_\psi}(|H_{\bar{f}}|^p,\mathcal{M}_{\psi,0}):=\inf_{A\in \mathcal{M}_{\psi,0}}\||H_{\bar{f}}|^p-A\|_{\mathcal{M}_\psi}\sim\limsup_{q\searrow p} \frac{1}{\psi(\mathrm{e}^{(q-p)^{-1}})}\|f\|_{B^{1/q}_{q,q},*}^q$$ 
\end{enumerate}
Moreover, if $\psi$ satisfies that $A_\psi(\alpha)\neq 1$ for some $\alpha>1$ (see Equation \eqref{apsicond}), there are holomorphic functions $f\in \cap_{q>p} B^{1/q}_{q,q}(S^1)$ such that $|H_{\bar{f}}|^p\in \mathcal{M}_\psi$ is non-measurable.
\end{thm}

Since we only consider $p$:th powers of operators, our results extend mutatis mutandis to $0<p<1$. We restrict our attention to $p\geq 1$ in order to avoid quasi-normed Banach spaces.

In Section \ref{prelsection} we provide an overview of the theory of Lorentz ideals from an extrapolation point of view. The general form of Theorem \ref{mainthmintro} will be considered in Section \ref{hanekopdldl}. We consider the special case $p=2,4,6$ of Theorem \ref{mainthmintro} in Section \ref{p246sec} where a result of Janson-Upmeier-Wallsten allows us to give exact formulas for the Dixmier trace $\mathrm{Tr}_{\omega,\psi}(|H_{\bar{f}}|^p)$. Finally, in Section \ref{Non-meas} we construct holomorphic functions $f\in \cap_{q>p} B^{1/q}_{q,q}(S^1)$ such that $|H_{\bar{f}}|^p\in \mathcal{M}_\psi$ is non-measurable.\\

\noindent {\bf Acknowledgements:}
We are grateful to Genkai Zhang for interesting discussions and particularly for introducing the work \cite{JUW} to us. We also thank Fedor Sukochev and Evgeniy Semenov for helpful comments on the distance formula to the separable part of a Lorentz ideal. We also acknowledge support from the Swedish Research Council Grant 2015-00137 and Marie Sklodowska Curie Actions, Cofund, Project INCA 600398. This work was finalized at the Erwin Schr\"odinger Institute in Vienna during the program on ``Bivariant $K$-theory in Geometry and Physics".

\section{Lorentz spaces and extrapolation}
\label{prelsection}

We will in this section provide an overview of Lorentz ideals and Hankel operators. Most results in this section can be found in the literature, and the remainder of the results in this section are variations on well known results.

\subsection{Operator ideals}
\label{opidelalsos}

The operators we will consider in this paper will in general belong to some ideal of operators on a Hilbert space. The general theory of operator ideals is well defined starting from a semi-finite von Neumann algebra. While this introduces some additional technicalities, it will allow us to treat ideal of operators on the same footing as $L^p$-spaces on a measure space. We will not go through this theory beyond its salient points. The reader is referred to \cite{LSZ} for a more thorough presentation.

Let $\mathcal{N}$ denote a semi-finite von Neumann algebra and $\tau$ a normal, faithful, semi-finite tracial weight on $\mathcal{N}$. The two main examples to keep in mind are $\mathcal{N}=\mathbb{B}(\mathcal{H})$ -- the bounded operators on a separable Hilbert space -- with $\tau$ being the operator trace and $\mathcal{N}=L^\infty(X,\mu)$ with $\tau(a):=\int_Xa\mathrm{d}\mu$ for a $\sigma$-finite measure space $(X,\mu)$. By definition, a von Neumann algebra is a weak operator closed $*$-subalgebra of $\mathbb{B}(\mathcal{H})$ for a Hilbert space $\mathcal{H}$. We will tacitly assume that $\mathcal{N}$ has a separable pre-dual, which is equivalent to $\mathcal{H}$ being separable. For any closed densely defined positive operator $T$ affiliated with $\mathcal{N}$, we define its singular value function 
\begin{align*}
\mu_T(t):&=\inf\{\|PT\|_{\mathcal{N}}: P\in \mathcal{N} \,\mbox{a projection with} \;\tau(1-P)\leq t\}=\\
&=\inf\{s\geq 0: \tau(\chi_{[s,\infty)}(T))\leq t\}
\end{align*}
There is a rich theory of so called symmetrically normed operator ideals, see more in \cite[Chapter 3]{LSZ}, which carries over to the theory of ideals in $L^\infty(0,\infty)$ by means of the singular value function. We are mainly interested in the following two classes. \\

{\bf $L^p$-spaces.} The noncommutative $L^p$-space $\mathcal{L}^p$ is defined as the set of operators affiliated with $\mathcal{N}$ such that $\mu_T\in L^p(0,\infty)$. The space $\mathcal{L}^p$ is a symmetrically normed operator ideal, in particular a Banach space, in the norm 
$$\|T\|_{\mathcal{L}^p}:=\|\mu_T\|_{L^p(0,\infty)}.$$

In the case that $\mathcal{N}=\mathbb{B}(\mathcal{H})$, we write $\mathcal{L}^p(\mathcal{H})$ for the associated noncommutative $L^p$-space. The space $\mathcal{L}^p(\mathcal{H})$ coincides with the $p$:th Schatten ideal with the same norm. 

In the case that $\mathcal{N}=L^\infty(X,\mu)$ with $\tau(a):=\int_Xa\mathrm{d}\mu$ for a $\sigma$-finite measure space $(X,\mu)$, it holds that $\mathcal{L}^p=L^p(X,\mu)$ with the same norm. \\

{\bf Lorentz ideals.} Let $\psi:[0,\infty)\to [0,\infty)$ be an increasing concave function with $\psi(0)=0$, $\lim_{t\to\infty} \psi(t)=\infty$. For later purposes of Dixmier trace computations, we often assume a condition which is slightly stronger than that in the original Dixmier paper. This condition is that the limit
\begin{equation}
\label{apsicond}
A_{\psi}(\alpha):=\lim_{t\to \infty} \frac{\psi(t^\alpha)}{\psi(t)}\quad\mbox{exists for all $\alpha>1$}. 
\end{equation}
Since $\psi$ is increasing, $A_\psi(\alpha)\geq 1$ for all $\alpha$. Condition \eqref{apsicond} guarantees that $\psi$ has regular variation of index $0$. Recall that a function $\psi$ has regular variation of index $\rho\in \R$ if 
$$\lim_{t\to\infty}\frac{\psi(\lambda t)}{\psi(t)}=\lambda^\rho, \quad\forall \lambda>0.$$
By \cite[Theorem 1.8.2]{RegVar} we can without restrictions assume that $\psi$ is smooth. For the purpose of extrapolation results, the following condition on $\psi$ often comes into play:
\begin{equation}
\label{Phi_as1}
\|\psi'\|_p \le C \psi(e^\frac1{p-1}), \ \forall p>1.
\end{equation}

We define the Lorentz ideal $\mathcal{M}_\psi$ to consist of operators affiliated with $\mathcal{N}$ such that 
$$\|T\|_{\mathcal{M}_\psi}:=\sup_{t>0}\frac{1}{\psi(t)}\int_0^t\mu_T(s)\mathrm{d}s<\infty.$$
The norm $\|\cdot\|_{\mathcal{M}_\psi}$ makes $\mathcal{M}_\psi$ into a symmetrically normed operator ideal. 

If the function $\psi$ satisfies condition~\eqref{apsicond}, the ideal $\mathcal{M}_\psi$ carries a plethora of singular traces, with Dixmier traces being those of most relevance to this paper. For $\alpha\geq 1$ we define $P_\alpha:L^\infty(0,\infty)\to L^\infty(0,\infty)$ by $P_\alpha f(t):=f(t^\alpha)$. If $\omega\in L^\infty(0,\infty)^*$ is a state satisfying that $\omega(f)=0$ if $\lim_{t\to \infty}f(t)=0$ we say that $\omega$ is an extended limit at $\infty$. By an abuse of notation, we write $\lim_{t\to\omega} f(t):=\omega(f)$ for an extended limit $\omega$ and $f\in L^\infty(0,\infty)$. If $\omega=\omega\circ P_\alpha$ for all $\alpha\geq 1$, we say that $\omega$ is an exponentiation invariant extended limit. Associated with an exponentiation invariant extended limit $\omega$ there is a Dixmier trace $\mathrm{Tr}_{\omega,\psi}:\mathcal{M}_\psi\to \C$ defined by 
$$\mathrm{Tr}_{\omega,\psi}(T):=\lim_{t\to \omega} \frac{1}{\psi(t)}\int_0^t \mu_T(t)\mathrm{d}t,$$
for positive $T\in \mathcal{M}_\psi$ and extending to $\mathcal{M}_\psi$ by linearity (see \cite[Proposition 1.12]{GS} for the proof).

The $p$:th convexification $\mathcal{M}_\psi^{(p)}$ is defined as the set of operators $T$ for which $|T|^p\in \mathcal{M}_\psi$; it is normed by $\|T\|_{\mathcal{M}_\psi^{(p)}}:=\||T|^p\|_{\mathcal{M}_\psi}^{1/p}$. The separable part $\mathcal{M}_{\psi,0}^{(p)}$ is defined as the closure in $\mathcal{M}_\psi^{(p)}$ of the finite trace operators in $\mathcal{N}$. 

The most studied example of Lorentz ideals comes from the function $\psi(t):=\log(1+t)$. In this case, one often writes $\mathcal{M}_{1,\infty}:=\mathcal{M}_\psi$ and $\mathcal{M}_{p,\infty}:=\mathcal{M}_\psi^{(p)}$. The reader should note that in \cite{EZ}, the Lorentz ideal $\mathcal{M}_{1,\infty}$ associated with $\mathcal{N}=\mathbb{B}(\mathcal{H})$ is denoted by $\mathcal{S}^{Dixm}$.

\subsection{Technical results on extrapolation and Dixmier traces}

The following result takes its starting point in work of Gayral-Sukochev \cite{GS}. The first statement is found in~\cite[Theorem 3.3]{GS} and the second statement in ~\cite[Proposition 2.17]{GS}. The third statement will be proven below, and is inspired by work of Engli\v{s}-Zhang \cite{EZ}.

\begin{thm}
\label{allthethingsGSsaid}
Let $\psi: [0,\infty)\to [0,\infty)$ be an increasing concave function satisfying \eqref{apsicond} and \eqref{Phi_as1},
$\psi(0)=0$, $\lim_{t\to\infty} \psi(t)=\infty$. We set $k_\psi := \log(A_{\psi}(e))$. 
\begin{enumerate} 
\item[(i)] For any exponentiation invariant extended limit $\omega\in (L^\infty)^*$ and $T\in \mathcal{M}_\psi^{(p)}$, the formula
$$\mathrm{Tr}_{\omega,\psi}(|T|^p)=\frac1{\Gamma(1+k_\psi)}\cdot\lim_{h\to  \tilde{\omega}} \frac{1}{\psi(\mathrm{e}^{1/h})}\|T\|_{p+h}^{p+h}$$
holds where $\tilde{\omega}\in (L_\infty)^*$ is the extended limit at $0$ given by 
\begin{equation}
\label{omegatildedef}
\lim_{h\to  \tilde{\omega}} x(t) := \lim_{t\to \omega} x(\frac1{\log(t)}).
\end{equation}
\item[(ii)] For any $T\in \mathcal{N}$,  
$$\|T\|_{ \mathcal{M}_\psi^{(p)}}\sim  \sup_{h>0} \frac{1}{\psi(\mathrm{e}^{1/h})}\|T\|_{p+h}^{p+h}.$$
In particular, $T\in  \mathcal{M}_\psi^{(p)}$ if and only if $\|T\|_{p+h}^{p+h}=O(\psi(\mathrm{e}^{1/h}))$.
\item[(iii)] Assume that 
$\mathcal{N}$ is atomic. For any $T\in \mathcal{M}_\psi^{(p)}$, we have that 
\begin{align*}
\limsup_{h\searrow 0} \frac{1}{\psi(\mathrm{e}^{1/h})}\|T\|_{p+h}^{p+h}&\leq \mathrm{d}_{\mathcal{M}_\psi}(|T|^p,\mathcal{M}_{\psi,0})\leq \mathrm{e} \limsup_{h\searrow 0} \frac{1}{\psi(\mathrm{e}^{1/h})}\|T\|_{p+h}^{p+h}.
\end{align*}
\end{enumerate}
\end{thm}

Before proving the third statement of this theorem, we need two lemmas. The following result is an extension of \cite[Proposition 7]{EZ}.

\begin{lem}
\label{prop7}
Let $\psi: [0,\infty)\to [0,\infty)$ be an increasing concave function satisfying the conditions \eqref{apsicond} and \eqref{Phi_as1} and moreover that
$\psi(0)=0$, $\lim_{t\to\infty} \psi(t)=\infty$. 
For a function $f\in \cap_{0<h<\delta}L^{p+h}(0,\infty)$ for some $\delta>0$ we define the quantities
$$\|f\|_{p,\limsup} := \limsup_{h\searrow0} \frac{\|f\|_{p+h}^{1+h}}{\sqrt[p]{\psi(\mathrm{e}^{\frac{1}{h}})}}\quad
\mbox{and} \quad 
\|f\|_{p, \lim \psi} := \limsup_{t\to \infty} \sqrt[p]{\frac{1}{\psi(t)}\int_0^t |f(s)|^p \mathrm{d}s}$$
It then holds that 
$$\|f\|_{p, \lim \psi}\leq \|f\|_{p,\limsup}\leq \mathrm{e} \|f\|_{p, \lim \psi}.$$
\end{lem}

\begin{proof}
If $\psi$ satisfies the conditions \eqref{apsicond} and \eqref{Phi_as1}, then so does $\psi^{1/p}$ for any $p\geq 1$. Indeed condition \eqref{apsicond} is readily verified for $\psi^{1/p}$ and condition \eqref{Phi_as1} for $\psi^{1/p}$ follows from that $\psi$ has regular variation of index $0$ and \cite[Proposition 2.17 and 2.23]{GS}. We can therefore replace $f$ by $H:=|f|^p\geq 0$ and $\psi$ by $\psi^p$, and thusly assume that $p=1$. For any $C>\|H\|_{1,\limsup}$ there is $q_0>0$ such that
$$\frac{\|H\|_{1+h}}{\psi(\mathrm{e}^\frac{1}{h})} <C, \quad\mbox{for}\; 0<h<q_0.$$

Using the H\"older inequality, for any $0<q<q_0$ we obtain
\begin{align*}
\int_0^t H(s) \mathrm{d} s &\le \left( \int_0^t H(s)^{1+q} \mathrm{d}s \right)^\frac1{1+q} \left( \int_0^t  \mathrm{d}s \right)^\frac{q}{1+q}\\
&\le C \cdot \psi(\mathrm{e}^\frac1{q})\cdot t^\frac{q}{1+q} \le C \cdot\psi(\mathrm{e}^\frac1{q})\cdot t^q.
\end{align*}

If $t>\mathrm{e}^{1/q_0}$ (that is, $q_0>1/\log t$), one can take $q=1/\log t$. Thus,
$$\int_0^t H(s) \mathrm{d}s \le C \mathrm{e} \cdot\psi(t), \quad\mbox{for}\;  t>\mathrm{e}^{1/q_0}.$$
Therefore,
$$\|H\|_{1,\lim \psi} \le \mathrm{e} \|H\|_{1,\limsup}.$$

Conversely for $C> \|H\|_{1,\lim \psi}$ there exists $t_0>0$ such  that
\begin{equation}
\label{eq1}
\frac1{\psi(t)} \int_0^t H(s) \mathrm{d}s \le C, \quad\forall t\ge t_0.
\end{equation}

Equivalently,
$$\int_0^t H(s) \mathrm{d}s \le \int_0^t C \psi'(s) \mathrm{d}s, \quad \forall t\ge t_0.$$

For the function 
$$G(t):=\begin{cases}
H(t), &t\ge t_0\\
\min\{H(t), C \psi'(t)\}, &t<t_0.
\end{cases}$$
we clearly have
$$\int_0^t G(s) \mathrm{d}s \le \int_0^t C \psi'(s) \mathrm{d}s,\quad \forall t> 0.$$

This means that the function $G$ is submajorised by the function $C \psi'$ (in the sense of Hardy-Littlewood). Thus, for every $h>0$ one has
$$\int_0^\infty G(s)^{1+h} \mathrm{d}s \le \int_0^\infty (C \psi'(s))^{1+h} \mathrm{d}s.$$
Since the function $\psi$ satisfies~\eqref{Phi_as1}, it follows that
$$\int_0^\infty G(s)^{1+h} \mathrm{d}s \le C^{1+h} (\psi(\mathrm{e}^{\frac{1}{h}}))^{1+h},$$
or, equivalently,
\begin{equation}
\label{G}
\limsup_{h\searrow0} \frac{\|G\|_{1+h}}{\psi(\mathrm{e}^{\frac{1}{h}})} \le C.
\end{equation}

First, 
$$\frac{1}{\psi(\mathrm{e}^{\frac{1}{h}})} \left(\int_0^{t_0} G(s)^{1+h} \mathrm{d}s\right)^{\frac{1}{1+h}} \le \frac{1}{\psi(\mathrm{e}^{\frac{1}{h}})} \left(\int_0^{t_0} (C \psi'(s))^{1+h} \mathrm{d}s\right)^{\frac{1}{1+h}} \mathop{\longrightarrow}\limits_{h\searrow0} 0,$$
since $\psi(\infty)=\infty$ and $\psi'\in L^{1+h}(0,\infty)$ for every $h>0$.

Second, by Lebesque Monotone Convergence Theorem and~\eqref{eq1} we obtain
$$\left(\int_0^{t_0} H(s)^{1+h} \mathrm{d}s\right)^{\frac{1}{1+h}} \mathop{\longrightarrow}\limits_{h\searrow0}\int_0^{t_0} H(s) \mathrm{d}s \le C \psi(t_0).$$

Therefore,
$$\limsup_{h\searrow0} \frac{1}{\psi(\mathrm{e}^{\frac{1}{h}})} \left(\int_0^{t_0} G(s)^{1+h} \mathrm{d}s\right)^{\frac{1}{1+h}} = \limsup_{h\searrow0}\frac{1}{\psi(\mathrm{e}^{\frac{1}{h}})} \left(\int_0^{t_0} H(s)^{1+h} \mathrm{d}s\right)^{\frac{1}{1+h}}=0.$$

Since $H(t)=G(t)$ for $t\ge t_0$, it follows from~\eqref{G} that
$$\limsup_{h\searrow0} \frac{\|H\|_{1+h}}{\psi(\mathrm{e}^{\frac{1}{h}})}\le C.$$
This proves that 
$$\|H\|_{1,\limsup} \le  \|H\|_{1,\lim \psi}.$$
\end{proof}

 The following result is well-known at least in the commutative setting (see e.g. \cite[Proposition 2.1]{DPSSS}. For the convenience of the reader we provide a short proof.

\begin{lem}
\label{prop7.5}
Let $\psi: [0,\infty)\to [0,\infty)$ be an increasing concave function satisfying the conditions \eqref{apsicond} and \eqref{Phi_as1} and moreover that
$\psi(0)=0$, $\lim_{t\to\infty} \psi(t)=\infty$. Assume that $\mathcal{N}$ is atomic. For any $T\in \mathcal{M}_\psi^{(p)}$, we have that 
$$\mathrm{d}_{\mathcal{M}_\psi}(|T|^p,\mathcal{M}_{\psi,0})= \limsup_{t\to \infty}  \frac1{\psi(t)}\int_0^t \mu_T(s)^p \mathrm{d}s.$$
\end{lem}

Let $\mathcal{M}'_{\psi,0}$ denote the norm closure of the space of elements $T\in \mathcal{M}_\psi$ with compactly supported singular value function. The assumption that $\mathcal{N}$ is atomic ensures that $\mathcal{M}_{\psi,0}=\mathcal{M}_{\psi,0}'$. Our proof will in fact consist of showing that for a general $\mathcal{N}$, it holds that 
\begin{equation}
\label{disisfifit}
\mathrm{d}_{\mathcal{M}_\psi}(|T|^p,\mathcal{M}_{\psi,0}')= \limsup_{t\to \infty}  \frac1{\psi(t)}\int_0^t \mu_T(s)^p \mathrm{d}s, \quad\forall T\in \mathcal{M}_\psi^{(p)}
\end{equation}
for any function $\psi$ additionally satisfying $\lim_{t\to0} \frac{t}{\psi(t)}=0$. Since the original statement is for atomic $\mathcal{N}$, we can always guarantee that this condition holds.

\begin{proof}
It follows from \cite{Chi_Suk} that for every $T\in \mathcal{M}_\psi^{(p)}$ there exists a rearrangement-preserving (and thus, isometric) embedding $i_T$ of $\mathcal{M}_\psi^{(p)}(0,\infty)$ into $\mathcal{M}_\psi^{(p)}$ such that $i_T(\mu(T))=T$. Thus, following the argument in \cite[Page 267]{CRSS}, it is sufficient to prove the formula \eqref{disisfifit} for every $x=\mu(x) \in \mathcal{M}_\psi^{(p)}(0,\infty)$.

For every $x=\mu(x) \in \mathcal{M}_\psi^{(p)}(0,\infty)$ and every $n\in\mathbb N$ the function $x^p\chi_{(0,n)} \in \mathcal{M}'_{\psi,0}(0,\infty)$. Hence, for every $n\in\mathbb N$ we have
$$\mathrm{d}_{\mathcal{M}_\psi(0,\infty)}(x^p,\mathcal{M}'_{\psi,0}(0,\infty))=\mathrm{d}_{\mathcal{M}_\psi(0,\infty)}(x^p\chi_{[n,\infty)},\mathcal{M}'_{\psi,0}(0,\infty)).$$

Therefore,
\begin{align*}
\mathrm{d}_{\mathcal{M}_\psi(0,\infty)}(x^p,\mathcal{M}'_{\psi,0}(0,\infty))
&\le\lim_{n\to\infty}\|x^p\chi_{[n,\infty)}\|_{\mathcal{M}_\psi(0,\infty)}\\
&=\lim_{n\to\infty}\sup_{t>0} \frac{1}{\psi(t)} \int_0^t \mu(x^p\chi_{[n,\infty)})(s) \ \mathrm{d}s\\
&=\lim_{n\to\infty}\sup_{t>0} \frac{1}{\psi(t)} \int_0^t (x(s+n))^p \ \mathrm{d}s\\
&=\lim_{n\to\infty}\sup_{t>0} \frac{1}{\psi(t)} \int_n^{t+n} (x(s))^p \ \mathrm{d}s.
\end{align*}

By the definition of supremum for every $n\in \N$ there exists $t_n>0$ such that 
$$\sup_{t>0} \frac{1}{\psi(t)} \int_n^{t+n} (x(s))^p \ \mathrm{d}s \le \frac{1}{\psi(t_n)} \int_n^{t_n+n} (x(s))^p \ \mathrm{d}s +\frac1n.$$

Denote for brevity
$$a:=\limsup_{t\to \infty}  \frac1{\psi(t)}\int_0^t (x(s))^p \mathrm{d}s.$$

1. If $\limsup_{n\to\infty} t_n=\infty$, then
$$\mathrm{d}_{\mathcal{M}_\psi(0,\infty)}(x^p,\mathcal{M}'_{\psi,0}(0,\infty))\le \lim_{n\to\infty}\sup_{t>n} \frac{1}{\psi(t)} \int_n^{t+n} (x(s))^p \ \mathrm{d}s\le a,$$
since $x=\mu(x)$.

2. If $0 < \liminf_{n\to\infty} t_n \le \limsup_{n\to\infty} t_n<\infty$, then
$$\frac{1}{\psi(t_n)} \int_n^{t_n+n} (x(s))^p \ \mathrm{d}s \le \frac{t_n (x(n))^p}{\psi(t_n)}\mathop{\longrightarrow}\limits_{n\to\infty} 0,$$
since $x=\mu(x)$ and $x(n)\to 0$ as $n\to\infty$.
Hence, $\mathrm{d}_{\mathcal{M}_\psi(0,\infty)}(x^p,\mathcal{M}'_{\psi,0}(0,\infty))=0\le a.$

3. If $0 = \liminf_{n\to\infty} t_n \le \limsup_{n\to\infty} t_n<\infty$, then
$$\frac{1}{\psi(t_n)} \int_n^{t_n+n} (x(s))^p \ \mathrm{d}s \le \frac{t_n (x(n))^p}{\psi(t_n)}\mathop{\longrightarrow}\limits_{n\to\infty} 0,$$
since $x$ is bounded and $\frac{t}{\psi(t)}\to 0$ as $t\to0$.
Hence, $\mathrm{d}_{\mathcal{M}_\psi(0,\infty)}(x^p,\mathcal{M}'_{\psi,0}(0,\infty))=0\le a.$

On the other hand, for every $x=\mu(x) \in \mathcal{M}_\psi^{(p)}(0,\infty)$ and $y\in \mathcal{M}'_{\psi,0}(0,\infty)$ by \cite[Theorem II.3.1]{KPS} we have
\begin{align*}
\mathrm{d}_{\mathcal{M}_\psi(0,\infty)}(x^p,\mathcal{M}'_{\psi,0}(0,\infty)) &=\|x^p-y\|_{\mathcal{M}_\psi(0,\infty)}\ge \|\mu(x^p)-\mu(y)\|_{\mathcal{M}_\psi(0,\infty)}\\
&\ge \limsup_{t\to \infty}  \frac1{\psi(t)}\int_0^t (\mu(x^p)-\mu(y))(s) \mathrm{d}s\\
&= \limsup_{t\to \infty}  \frac1{\psi(t)}\int_0^t \mu(x^p)(s) \mathrm{d}s,
\end{align*}
since $y\in \mathcal{M}'_{\psi,0}(0,\infty)$. This proves the assertion.
\end{proof}

\begin{proof}[Proof of third statement in Theorem \ref{allthethingsGSsaid}]
Set $f=\mu_T$. Assuming that $\mathcal{N}$ is atomic, Lemma \ref{prop7.5} ensures that $\mathrm{d}_{\mathcal{M}_\psi}(|T|^p,\mathcal{M}_{\psi,0})= \|f\|_{p,\lim \psi}^p$. By definition, 
$$\limsup_{h\searrow 0} \frac{1}{\psi(\mathrm{e}^{1/h})}\|T\|_{p+h}^{p+h}=\|f\|_{p,\limsup}^p.$$ 
We conclude the inequality stated in the third statement of Theorem \ref{allthethingsGSsaid} from Lemma \ref{prop7}.
\end{proof}

The aspect of Theorem \ref{allthethingsGSsaid} relevant to this paper lies in its implications on Hankel operators. To formalize this, we state an immediate corollary of Theorem \ref{allthethingsGSsaid}. If $(X_h)_{h\in [0,1]}$ is a family of Banach spaces with $X_h\subseteq X_{h'}$ continuously for $h<h'$, we define the extrapolation space $X_\psi\subseteq \cap_{h\in (0,1]}X_h$ to be the set of all elements $x\in\cap_{h\in (0,1]}X_h$ for which 
$$\|x\|_{X_\psi}:=\sup_{h>0} \frac{1}{\psi(\mathrm{e}^{1/h})}\|x\|^{1+h}_{X_h}<\infty.$$

\begin{cor}
\label{equivalencescor}
Let $\psi: [0,\infty)\to [0,\infty)$ be an increasing concave function satisfying the conditions \eqref{apsicond} and \eqref{Phi_as1} and moreover that
$\psi(0)=0$, $\lim_{t\to\infty} \psi(t)=\infty$. Consider the following data:
\begin{itemize}
\item A family of Banach spaces $(X_h)_{h\in [0,1]}$ with $X_h\subseteq X_{h'}$ continuously for $h<h'$.
\item A mapping $T:X_1\to \mathcal{L}^{p+1}$ restricting to a continuous mapping $T_h:=T|_{X_h}:X_h\to \mathcal{L}^{p+h}$, for $h\in [0,1]$, such that there are measurable functions 
$$c_0,c_1:[0,1]\to [r,R],\quad \mbox{for some $0<r\leq R<\infty$},$$ 
with 
$$c_0(h)\|x\|_{X_h}\leq \|T_h(x)\|_{\mathcal{L}^{p+h}}\leq c_1(h)\|x\|_{X_h},\quad \forall h\in [0,1], \; x\in X_h.$$
\end{itemize} 
Then $T$ defines a continuous mapping $T:X_\psi \to \mathcal{M}_\psi^{(p)}$ such that 
\begin{enumerate} 
\item[A)] For any exponentiation invariant extended limit $\omega\in (L_\infty)^*$ 
$$\lim_{h\to  \tilde{\omega}} \frac{c_0(h)^{p}}{\psi(\mathrm{e}^{1/h})}\|x\|_{X_{h}}^{p+h}\leq \mathrm{Tr}_{\omega,\psi}(|T(x)|^p)\leq \lim_{h\to  \tilde{\omega}} \frac{c_1(h)^p}{\psi(\mathrm{e}^{1/h})}\|x\|_{X_{h}}^{p+h},$$
where $\tilde{\omega}$ is defined as in Equation \eqref{omegatildedef}.
In particular, if $\lim_{h\to 0}\frac{c_0(h)}{c_1(h)}=1$, then 
$$\mathrm{Tr}_{\omega,\psi}(|T(x)|^p)= \lim_{h\to  \tilde{\omega}} \frac{c_0(h)^p}{\psi(\mathrm{e}^{1/h})}\|x\|_{X_{h}}^{p+h}.$$
\item[B)] For any $x\in X_\psi$ we have that 
$$r\|x\|_{X_\psi}\leq \|T(x)\|_{\mathcal{M}_\psi^{(p)}}\leq R\|x\|_{X_\psi}.$$
\item[C)] Assume that $\mathcal{N}$ is atomic. For any $x\in X_\psi$ we have that 
$$r\limsup_{h\searrow 0} \frac{1}{\psi(\mathrm{e}^{1/h})}\|x\|^{p+h}_{X_h}\leq \mathrm{d}_{\mathcal{M}_\psi^{(p)}}(|T(x)|^p,\mathcal{M}_{\psi,0})\leq \mathrm{e}R\limsup_{h\searrow 0} \frac{1}{\psi(\mathrm{e}^{1/h})}\|x\|^{p+h}_{X_h}.$$
\end{enumerate}
\end{cor}

\begin{rem}
In the setup of Corollary \ref{equivalencescor}, we note that the norms $\|x\|_{X_h}':=\|T(x)\|_{\mathcal{L}^{p+h}}$ on $X_h$ are equivalent to the norms $\|\cdot \|_{X_h}$. After this change of norms, we can take $c_0=c_1=1$ in which case Corollary \ref{equivalencescor} is a trivial reformulation of Theorem \ref{allthethingsGSsaid}. \emph{The relevance of Corollary \ref{equivalencescor} lies in that it is often possible to estimate the norms $\|x\|_{X_h}$ in situations where it is not possible to estimate $\|T(x)\|_{\mathcal{L}^{p+h}}$.} We will utilize this fact below for Hankel operators. 
\end{rem}

\begin{rem}
In part A of Corollary \ref{equivalencescor}, we can obtain equivalences that are independent of $\omega$. Indeed the upper and lower bounds on $c_0$ and $c_1$ implies that under the assumptions of of Corollary \ref{equivalencescor}, 
$$r\lim_{h\to  \tilde{\omega}} \frac{1}{\psi(\mathrm{e}^{1/h})}\|x\|_{X_{h}}^{p+h}\leq \mathrm{Tr}_{\omega,\psi}(|T(x)|^p)\leq R\lim_{h\to  \tilde{\omega}} \frac{1}{\psi(\mathrm{e}^{1/h})}\|x\|_{X_{h}}^{p+h}$$
\end{rem}

\section{Hankel operators and Peller's characterization}
\label{hanekopdldl}

We now turn our focus to Hankel operators on the Hardy space. The reader can recall that the Hardy space $H^2(S^1)\subseteq L^2(S^1)$ is defined as the subspace of functions with a holomorphic extension to the interior of the unit disc. We here consider $S^1$ to be the boundary of the unit disc in the complex plane. The orthogonal projection $P:L^2(S^1)\to H^2(S^1)$ is called the Szeg\"o projection. For $f\in L^\infty(S^1)$, the associated Hankel operator is defined as 
$$H_f:=(1-P)fP.$$
Clearly, if $f$ is the restriction of holomorphic function in the unit disc, $H_f=0$. In fact, $H_f$ is a well defined bounded operator for $f\in BMO(S^1)$. The space of symbols $f$ for which $H_f\in \mathcal{L}^p(L^2(S^1))$ has been characterized in terms of Besov spaces by Peller \cite{peller}. We let $B^{1/p}_{p,p}(S^1)$ denote the Besov space on $S^1$, we will review this space and various equivalent norms on this space below. 

For now we fix a particular choice of norms on the scale of Besov space defined in terms of Littlewood-Paley theory. Let $W_0:=1$ and for $n\in \N_+$ we define 
$$W_n(z):=\sum_{k=2^{n-1}}^{2^{n+1}} \min\left(\frac{k-2^{n-1}}{2^n-2^{n-1}},\frac{2^{n+1}-k}{2^{n+1}-2^{n}}\right) (z^k+z^{-k}).$$
The polynomials $W_n$ are characterized by the property that their Fourier coefficients $(\hat{W}_n(k))_{k\in \Z}$ is linearly interpolating between $\hat{W}_n(2^{n-1})=\hat{W}_n(2^{n+1})=0$, $\hat{W}_n(2^n)=1$ and $\hat{W}_n(k)=\hat{W}_n(-k)$. In particular, $\sum_{n=0}^\infty \hat{W}_n(k)=1$ for any $k$. For a function $f$ on $S^1$, we define 
$$\Phi_n(f):=W_n*f.$$
A well known result from Littlewood-Paley theory guarantees that for any function $f$ on $S^1$, 
$$\|f\|_{L^p(S^1)}\sim \|(\Phi_n f)_{n\in \N}\|_{L^p(S^1\times \N)}.$$

\begin{dfn}
We define 
$$\|f\|_{B^{1/p}_{p,p}(S^1)}:=\|(2^{n/p}\Phi_n f)_{n\in \N}\|_{L^p(S^1\times \N)}.$$
\end{dfn}

\begin{thm}[Peller \cite{peller}]
\label{pellerhem}
Let $f$ be a function on $S^1$ extending holomorphically to the unit disc with $f(0)=0$. Then $H_{\overline{f}}\in \mathcal{L}^p(L^2(S^1))$ if and only if $f\in B^{1/p}_{p,p}(S^1)$. Moreover, for any $p_0>1$ there is a constant $C>0$ such that
$$C^{-1}\|f\|_{B^{1/p}_{p,p}(S^1)}\leq \|H_{\overline{f}}\|_{\mathcal{L}^p(L^2(S^1))}\leq C\|f\|_{B^{1/p}_{p,p}(S^1)}, \quad \forall p\in [1,p_0].$$
\end{thm}

The reader can note that the statement in \cite[Chapter 6.2, Theorem 2.1]{peller} does not give a uniform constant, but existence of a uniform constant follows from the fact that the proof is by interpolation. We shall use Peller's theorem to compute and estimate Dixmier traces. To do so, it will be important to keep track of the norms used on the Besov spaces. Let us state a general result regarding the estimates of Dixmier traces of Hankel operators. This statement is a direct consequence of Corollary \ref{equivalencescor} and Theorem \ref{pellerhem}. 

\begin{cor}
\label{equivbesovgeneerl}
Let $p\geq 1$,  and $\psi:$ be a function as in Corollary \ref{equivalencescor}. Assume that $(\|\cdot\|_{B^{1/q}_{q,q},*})_{q\geq p}$ is a family of norms on the Besov spaces $B^{1/q}_{q,q}(S^1)$ for $q\geq p$ such that there is a $p_0>p$ and a constant $C_0>0$ such that 
$$C_0^{-1} \|f\|_{B^{1/q}_{q,q},*}\leq \|f\|_{B^{1/q}_{q,q}}\leq C_0 \|f\|_{B^{1/q}_{q,q},*}, \quad \forall q\in [p,p_0].$$
Then for any holomorphic function $f$, 
$$H_{\bar{f}}\in \mathcal{M}_\psi^{(p)}\quad\Leftrightarrow \sup_{q>p}\frac{\|f\|_{B^{1/q}_{q,q},*}^q}{\psi(\mathrm{e}^{(q-p)^{-1}})} <\infty.$$ 
Moreover, there is a constant $C>0$ (independent of $f$) such that for any exponentiation invariant extended limit $\omega$, 
$$C^{-1}\lim_{q-p\to  \tilde{\omega}}\frac{1}{\psi(\mathrm{e}^{(q-p)^{-1}})} \|f\|_{B^{1/q}_{q,q},*}^q\leq \mathrm{Tr}_{\omega,\psi}(|H_{\bar{f}}|^p)\leq C\lim_{q-p\to  \tilde{\omega}}\frac{1}{\psi(\mathrm{e}^{(q-p)^{-1}})} \|f\|_{B^{1/q}_{q,q},*}^q.$$
Finally, for any holomorphic function $f\in \cap_{q>p} B^{1/q}_{q,q}(S^1)$
$$\mathrm{d}_{\mathcal{M}_\psi}(|H_{\bar{f}}|^p,\mathcal{M}_{\psi,0})\sim\limsup_{q\searrow p} \frac{1}{\psi(\mathrm{e}^{(q-p)^{-1}})}\|f\|_{B^{1/q}_{q,q},*}^{q}$$
\end{cor}

Corollary \ref{equivbesovgeneerl} can be applied to a variety of different norms on the scale of Besov spaces. Let $f$ be a function on $S^1$ extending holomorphically to the unit disc and $p\in [1,\infty)$. By an abuse of notation, we identify $f$ with its holomorphic extension $f:\mathbb{D}\to \C$, where $\mathbb{D}$ denotes the unit disc. Let $\mu$ denote the measure on $\mathbb{D}$ given by $\mathrm{d}\mu(z)=(1-|z|^2)^{-2}\mathrm{d}m(z)$ where $m$ denotes the Lebesgue measure. 

\begin{dfn}
We define 
$$\|f\|_{B^{1/p}_{p,p},\mathbb{D}}:=\|(1-|z|^2)^2f''\|_{L^p(\mathbb{D},\mu)}.$$
\end{dfn}

The next result can also be found in \cite[Appendix 2.6]{peller}.

\begin{prop}
For any $p_0>1$ there is a constant $C>0$ such that for all holomorphic $f$ with $f(0)=f'(0)=0$
$$C^{-1}\|f\|_{B^{1/p}_{p,p},\mathbb{D}}\leq \|f\|_{B^{1/p}_{p,p}(S^1)}\leq C\|f\|_{B^{1/p}_{p,p},\mathbb{D}}, \quad \forall p\in [1,p_0].$$
\end{prop}

We remark that the condition $f(0)=f'(0)=0$ plays no role once going to the extrapolation space because we can write any $f=f_0+g_0$ where $f_0(0)=f'_0(0)=0$ and $g_0=-f'(0)z-f(0)$ satisfies that $H_{\bar{g}_0}$ is finite rank.

For a holomorphic $f\in \cap_{q>p} B^{1/q}_{q,q}(S^1)$ we define $F_f\in \cap_{q>p}L^q(0,\infty)$ as the decreasing rearrangement of the function $(1-|z|^2)^2f''$ on $\mathbb{D}$ with respect to the measure $\mu$. We also define $\Phi_f\in \cap_{q>p}L^q(0,\infty)$ as the decreasing rearrangement of the function $S^1\times \N\ni (\theta,n)\mapsto W_n*f(\mathrm{e}^{i\theta})$ with respect to the product measure $\nu$ on $S^1\times \N$. It is follows from the well known fact that $L^q$-norms of functions coincides with the $L^q(0,\infty)$-norm of its decreasing rearrangement that for $q>p$
\begin{align*}
\|f\|_{B^{1/q}_{q,q}(S^1)}&=\|W_n*f\|_{L^q(S^1\times \N,\nu)}=\|\Phi_f\|_{L^q(0,\infty)}\quad \mbox{and}\\
\|f\|_{B^{1/q}_{q,q},\mathbb{D}}&=\|(1-|z|^2)^2f''\|_{L^p(\mathbb{D},\mu)}=\|F_f\|_{L^q(0,\infty)}.
\end{align*}

\begin{thm}
\label{Thm1}
Let $p\geq 1$ and $\psi$ be a function as in Corollary \ref{equivalencescor}. Assume that $f$ is holomorphic. Then the following assertions are equivalent:
\begin{enumerate}
\item $\limsup_{h\searrow 0} \left(\frac{1}{\psi(\mathrm{e}^{\frac{1}{h}})} \int_{\textbf{D}} |f''(z)|^{p+h} (1-|z|^2)^{2p+2h-2} dz\right)^{\frac{p}{p+h}} <\infty;$
\item $\limsup_{t\to \infty} \frac1{\psi(t)} \int_0^t F_f(s)^p ds < \infty;$
\item $\limsup_{h\searrow 0} \left(\frac{1}{\psi(\mathrm{e}^{\frac{1}{h}})} \int_{\textbf{T}\times \textbf{N}} |f *W_n)(e^{i\theta})|^{p+h} d\nu(\theta, n)\right)^{\frac{p}{p+h}} <\infty;$
\item $\limsup_{t\to \infty} \frac1{\psi(t)} \int_0^t \Phi_f(s)^p ds < \infty;$
\item  $H_{\overline{f}} \in \mathcal{M}_\psi^{(p)}.$
\end{enumerate}
The quantities in (1)-(4) are all equivalent to $\mathrm{d}_{\mathcal{M}_\psi}(|H_{\overline{f}}|^p,\mathcal{M}_{\psi,0})$.
\end{thm}

The proof of this result is a straightforward repetition of that of \cite[Theorem 1]{EZ} with the replacement of $\log t$, $p-1$ and the use of \cite[Proposition 7]{EZ} by $\psi(t)$, $\psi(e^\frac1{h})$ and the use of Proposition~\ref{prop7}, respectively. Again, as in Proposition~\ref{prop7} we can reduce the proof to $p=1$.

\begin{prop}\label{prop8}
Let $\psi$ satisfy~\eqref{apsicond}. 
Let $H=H^*\in L^q(0,\infty)$ for all $1<q<1+\delta$ for some $\delta>0$. Let $\omega$ be an exponentiation invariant extended limit on $L^\infty(0,\infty)$ and $\hat{\omega}:=\omega \circ {\rm exp}$.

(a) For every $\alpha>1$ and sufficiently large $t>0$ one has
$\mu_H(1/t) \le t^\alpha;$

(b) One has
$$\lim_{t\to \omega} \frac1{\psi(t)} \int_0^t H(s) ds = \lim_{t\to \omega} \frac1{\psi(t)} \int_0^{\mu_H(1/t)} H(s) ds;$$

(c) One has
$$\lim_{r\to \hat{\omega}} \frac{\|H\|_{1+1/r}}{\psi(\mathrm{e}^{r})}= \lim_{t\to \omega} \frac1{\psi(t)} \int_0^t H(s) ds.$$
\end{prop}

\begin{proof}

(a) Denote for brevity $a:=\mu_H(1/t).$ For sufficiently large $t>0$ we have
$$c_H := \sup_{t>2} \frac1{\psi(t)} \int_0^t H(s) ds \ge \frac1{\psi(a)} \int_0^a H(s) ds \ge \frac{a H(a)}{\psi(a)}= \frac{a}{t \psi(a)},$$
since $H$ is nonincreasing and $H(\mu_H(1/t))=1/t$. Since the function $\psi$ is slowly varying, it follows that for every $0<\varepsilon <1$ there exists $C>0$ such that $\psi(t)\le C t^\varepsilon$ for all $t>0$. Hence,
$$c_H  \ge \frac{a}{t C a^\varepsilon}=\frac{a^{1-\varepsilon}}{C t}.$$

Therefore,
$$\mu_H(1/t) \le (C c_H t)^\frac{1}{1-\varepsilon}.$$

Since this inequality holds for every $0<\varepsilon <1$, it follows that for every $\alpha>1$ and sufficiently large $t>0$ one has
$\mu_H(1/t) \le t^\alpha.$

(b) For sufficiently large $t>0$ one has
$$\int_0^t H(s) ds \le \int_0^{\mu_H(1/t)} H(s) ds +1 \le \int_0^{t^\alpha} H(s) ds+1,$$
where the first inequality was proved in~\cite[Proposition 8]{EZ} and the second one was proved above.

Dividing by $\psi(t)$ and applying extended limits, yields
\begin{equation}\label{eq2}
\lim_{t\to \omega} \frac1{\psi(t)} \int_0^t H(s) ds\le\lim_{t\to \omega} \frac1{\psi(t)} \int_0^{\mu_H(1/t)} H(s) ds \le\lim_{t\to \omega} \frac1{\psi(t)} \int_0^{t^\alpha} H(s) ds.
\end{equation}

Since $\psi$ satisfies~\eqref{apsicond}, it follows from the property of extended limits that
\begin{align*}
\lim_{t\to \omega} \frac1{\psi(t^\alpha)} \int_0^{t^\alpha} H(s) ds&\le \lim_{t\to \omega} \frac{\psi(t^\alpha)}{\psi(t)} \frac1{\psi(t^\alpha)} \int_0^{t^\alpha} H(s) ds\\ 
&=A_\psi(\alpha) \lim_{t\to \omega}  \frac1{\psi(t^\alpha)} \int_0^{t^\alpha} H(s) ds.
\end{align*}

Since $\omega$ is exponentiation invariant, it follows that
\begin{equation}\label{eq3}
\lim_{t\to \omega} \frac1{\psi(t^\alpha)} \int_0^{t^\alpha} H(s) ds\le A_\psi(\alpha) \lim_{t\to \omega}  \frac1{\psi(t)} \int_0^{t} H(s) ds.
\end{equation}

Combining~\eqref{eq2} and~\eqref{eq3}, we obtain that
$$\lim_{t\to \omega} \frac1{\psi(t)} \int_0^t H(s) ds\le\lim_{t\to \omega} \frac1{\psi(t)} \int_0^{\mu_H(1/t)} H(s) ds \le A_\psi(\alpha) \lim_{t\to \omega}  \frac1{\psi(t)} \int_0^{t} H(s) ds$$
holds for every $\alpha>1$. It follows from \cite[Lemma 1.3]{GS} that $A_\psi(\alpha)\to 1$ as $\alpha \searrow 1$. This proves part (b).

(c) The proof of part (c) is a straightforward repetition of \cite[Proposition 8 (c)]{EZ} with the only difference that instead of the classical Karamata theorem one has to use its generalisation proved in \cite[Proposition 3.2]{GS}.
\end{proof}

\begin{thm}
\label{Thm2} 
Let $p\geq 1$, $\psi$ be a function as in Corollary \ref{equivalencescor} and $\omega$ an exponentiation invariant extended limit on $L^\infty(0,\infty)$. Assume that $f\in \cap_{q>p} B^{1/q}_{q,q}(S^1)$ is holomorphic.
 The following quantities are equivalent:
\begin{enumerate}
\item $$\lim_{h\to\tilde{\omega}} \left(\frac{1}{\psi(\mathrm{e}^{\frac{1}{h}})} \int_{\textbf{D}} |f''(z)|^{p+h} (1-|z|^2)^{2p+2h-2} dz\right)^{\frac{p}{p+h}} = \lim_{t\to \omega} \frac1{\psi(t)} \int_0^t F_f(s)^p ds;$$
\item $$\lim_{h\to\tilde{\omega}} \left(\frac{1}{\psi(\mathrm{e}^{\frac{1}{h}})} \int_{\textbf{T}\times \textbf{N}} |f *W_n)(e^{i\theta})|^{p+h} d\nu(\theta, n)\right)^{\frac{p}{p+h}}  = \lim_{t\to \omega} \frac1{\psi(t)} \int_0^t \Phi_f(s)^p ds;$$
\item ${\rm Tr}_{\omega,\psi} |H_{\overline{f}}|^p.$\\
\end{enumerate}
Here $\tilde{\omega}$ is defined as in Equation \eqref{omegatildedef}.
\end{thm}

The proof of this result is a straightforward repetition of that of \cite[Theorem 2]{EZ} with the replacement of $\log t$, $1/r$ and the use of \cite[Proposition 8]{EZ} by $\psi(t)$, $\psi(\mathrm{e}^{1/h})$ and the use of Proposition~\ref{prop8}, respectively.

Let us place the result Theorem \ref{Thm2} in context. Let $B^{1/q+}_{q,q}(S^1)$ denote the subspace of $B^{1/q}_{q,q}(S^1)$ consisting of holomorphic functions. By the results above, we can define two continuous linear mappings
\begin{align*}
T_{LP}:B^{1/q+}(S^1)&\to L^q(S^1\times \N,\nu), \quad T_{LP}f(z,n):=W_n*f(z),\quad\mbox{and}\\
T_{\mathbb{D}}:B^{1/q+}(S^1)&\to L^q(\mathbb{D},\mu), \quad T_{\mathbb{D}}f(z):=(1-|z|^2)^2f''(z).
\end{align*}
We define the spaces $\mathcal{M}_\psi^{(p)}(S^1\times \N,\nu)\subseteq \cap_{q>p} L^q(S^1\times \N,\nu)$ and $\mathcal{M}_\psi^{(p)}(\mathbb{D},\mu)\subseteq \cap_{q>p} L^q(\mathbb{D},\mu)$ from the families $(L^q(S^1\times \N,\nu))_{q>p}$ and $(L^q(\mathbb{D},\mu)_{q>p}$, respectively, by means of extrapolation. For any exponentiation invariant extended limit $\omega$, we can define Dixmier traces $\mathrm{tr}_{\omega,\psi}:\mathcal{M}_\psi(S^1\times \N,\nu)\to \C$ and $\mathrm{tr}_{\omega,\psi}:\mathcal{M}_\psi(\mathbb{D},\mu)\to \C$. We write $\mathrm{tr}_{\omega,\psi}$ to emphasize that the Dixmier trace is defined on a commutative von Neumann algebra. Applying Corollary \ref{equivalencescor}, we can reformulate Theorem \ref{Thm2} as the statement that 
$${\rm Tr}_{\omega,\psi} |H_{\overline{f}}|^p\sim \mathrm{tr}_{\omega,\psi}(|T_{LP}f|^p)\sim\mathrm{tr}_{\omega,\psi}(|T_{\mathbb{D}}f|^p) .$$

\section{The special case $p=2,4,6$}
\label{p246sec}

A beautiful result of Janson-Upmeier-Wallst\'en \cite{JUW} computes the operator trace of $|H_{\overline{f}}|^p$ for $p=2,4,6$ in terms of a particular Besov norm. Indeed, \cite[Theorem 1]{JUW} states that for $p=2,4,6$ and $f$ holomorphic in $\mathbb{D}$ we have that 
\begin{equation}
\label{p246}
\mathrm{Tr}(|H_{\overline{f}}|^p)=c_p\int_{S^1\times S^1} \frac{|f(z)-f(w)|^{p}}{|z-w|^2}\mathrm{d}V(z,w),
\end{equation}
where $c_2=1$, $c_4=\frac{1}{2}$ and $c_6=\frac{1}{6}$. In fact, \cite[Theorem 1]{JUW} states that $p=2,4,6$ are the only possible values for which an identity of this type could hold true.

\begin{dfn}
\label{SIdefn}
For $p>1$ we define 
$$\|f\|_{B^{1/p}_{p,p},SI}:=\left(\int_{S^1\times S^1} \frac{|f(z)-f(w)|^p}{|z-w|^2}\mathrm{d}V(z,w)\right)^{1/p}.$$
\end{dfn}

The reader should note that Equation \eqref{p246} is equivalent to the identity 
$$\|H_{\overline{f}}\|_{\mathcal{L}^p(L^2(S^1))}^p=c_p\|f\|_{B^{1/p}_{p,p},SI}^p, \quad\mbox{for $p=2,4,6$}.$$ 
The following norm equivalence is found in \cite[Appendix 2.6]{peller}.

\begin{prop}
\label{sivsdnorm}
For any $p_0\geq q_0>1$ and there is a constant $C>0$ such that 
$$C^{-1}\|f\|_{B^{1/p}_{p,p}}\leq \|f\|_{B^{1/p}_{p,p},SI}\leq C\|f\|_{B^{1/p}_{p,p}}, \quad \forall p\in [q_0,p_0].$$
\end{prop}

The result of Janson-Upmeier-Wallst\'en together with Theorem \ref{pellerhem} and Proposition \ref{sivsdnorm} allow us to deduce the following proposition.

\begin{prop}
\label{equiforjuw}
There are constants $0<r< R<\infty$ and measurable functions $c_0,c_1:[3/2,8]\to [r,R]$ such that 
$$c_0(p)\|f\|_{B^{1/p}_{p,p},SI}\leq \|H_{\overline{f}}\|_{\mathcal{L}^p(L^2(S^1))}\leq c_1(p)\|f\|_{B^{1/p}_{p,p},SI}.$$
Moreover, we can choose $c_0$ and $c_1$ so that
$$\lim_{h\to 0} c_0(p+h)^{\frac{1}{p}}=\lim_{h\to 0} c_1(p+h)^{\frac{1}{p}}=c_p\quad\mbox{for $p=2,4,6$}.$$
\end{prop}

For $p>1$ and the scale of spaces $(B^{1/q}_{q,q}(S^1))_{q\in [p,p+1]}$ we let $B_{p,\psi}(S^1)$ denote the associated extrapolation space. Using Corollary \ref{equivalencescor}, we deduce the following theorem from Proposition \ref{equiforjuw}.

\begin{thm}
For $p=2,4,6$, and a holomorphic $f\in B_{p,\psi}(S^1)$, we have that 
$$\mathrm{Tr}_{\omega,\psi}(|H_f|^p)= c_p\lim_{h\to  \tilde{\omega}} \frac{1}{\psi(\mathrm{e}^{1/h})}\int_{S^1\times S^1} \frac{|f(z)-f(w)|^{p+h}}{|z-w|^2}\mathrm{d}V(z,w),$$
where $c_2=1$, $c_4=\frac{1}{2}$ and $c_6=\frac{1}{6}$.
\end{thm}

\begin{rem}
The special case $p=2$ and $f\in C^{1/2}(S^1)$ was considered in \cite{GG}, where explicit formulas for $\mathrm{Tr}_\omega(|H_{\overline{f}}|^2)$ was given in terms of the Fourier series of $f$.
\end{rem}

\section{Non-measurability}
\label{Non-meas}

The estimates for Dixmier traces will allow us to construct an abundance of non-measurable Hankel operators by means of lacunary Fourier series. Our approach is based on results from \cite[Section 6]{EZ}. For $p\in [1,\infty)$ and $c\in \ell^\infty(\N)$ we define the function $\gamma_{p,c}$ on $S^1$ by 
$$\gamma_{p,c}(z):=\sum_{j=0}^\infty 2^{-j/p}c_j z^{2^j}.$$
The function $\gamma_{p,c}$ is holomorphic in $\mathbb{D}$. We can compute that 
$$\Phi(t)=2^{-j/p} c_j, \quad t\in [2^j-1,2^{j+1}-1).$$
Therefore, $\|\gamma_{p,c}\|_{B^{1/p}_{p,p}}\sim \|\Phi\|_{L^p(0,\infty)}\sim \|c\|_{\ell^p(\N)}$. Moreover, we can as in \cite[Page 351]{EZ} compute that for $t\in [2^j-1,2^{j+1}-1)$
\begin{equation}
\label{equistuff}
\frac{\sum_{k=0}^{j-1} |c_k|^p}{\psi(2^j-1)}\lesssim \frac{1}{\psi(t)}\int_0^t \Phi(t)^p\mathrm{d}t\lesssim\frac{\sum_{k=0}^{j} |c_k|^p}{\psi(2^{j+1}-1)}.
\end{equation}

Define the function $\tilde{\psi}(t):=\psi(2^t-1)$. This is again an increasing function with $\tilde{\psi}(0)=0$ and $\lim_{t\to \infty}\tilde{\psi}(t)=\infty$. If $\psi$ satisfies \eqref{apsicond}, then $\tilde{\psi}$ has regular variation of index $k_\psi:=\log A_\psi(e)$. We write $\mathfrak{m}_{\tilde{\psi}}^{(p)}(\N):=\mathcal{M}_{\tilde{\psi}}^{(p)}(\ell^\infty(\N))$. The inequalities \eqref{equistuff} imply that 
$$\|c\|_{\mathfrak{m}_{\tilde{\psi}}^{(p)}}^p\sim \sup_{t>0}\frac{1}{\psi(t)}\int_0^t \Phi(t)^p\mathrm{d}t\sim \|\gamma_{p,c}\|_{B_{p,\psi}}^p,$$ 
so the map $c\mapsto \gamma_{p,c}$ defines a continuous mapping 
$$\gamma:\mathfrak{m}_{\tilde{\psi}}^{(p)}(\N)\to B_{p,\psi}.$$

It follows from Theorem \ref{Thm2} and the inequalities \eqref{equistuff} that for any exponentiation invariant extended limit $\omega$ we have
\begin{align*}
\mathrm{Tr}_{\omega,\psi}(|H_{\overline{\gamma_{p,c}}}|^p) &\sim \lim_{t\to \omega} \frac{1}{\psi(t)}\int_0^t \Phi(t)^p\mathrm{d}t \sim \lim_{t\to \omega} \frac{\sum_{k=0}^{\lfloor \log_2 t \rfloor} |c_k|^p}{\psi(2^{\lfloor \log_2 t \rfloor+1}-1)} \\
&= \lim_{t\to \omega\circ \log_2} \frac{\sum_{k=0}^{\lfloor  t \rfloor} |c_k|^p}{\psi(2^{\lfloor t \rfloor+1}-1)}=\lim_{t\to \omega\circ \log_2} \frac{\sum_{k=0}^{\lfloor  t \rfloor} |c_k|^p}{\tilde\psi( t )} .
\end{align*}
Denote
\begin{align}
\label{small_DT}
\mathrm{tr}_{\omega\circ \log_2,\tilde{\psi}}(x):= \lim_{t\to\omega\circ \log_2} \frac{\sum_{k=0}^{\lfloor  t \rfloor} x_k}{\tilde\psi( t )}, \quad x \in \mathfrak{m}_{\tilde{\psi}}(\N)_+.
\end{align}
Here we note that $\mathrm{tr}_{\omega\circ \log_2,\tilde{\psi}}$ extends to a singular linear functional on $\mathfrak{m}_{\tilde{\psi}}(\N)$ because it is the Dixmier trace $\mathrm{tr}_{\omega,\psi}$ on $\mathfrak{m}_{\psi}(\N)$ pulled back along the isometric order preserving embedding $\mathfrak{m}_{\tilde{\psi}}(\N)\hookrightarrow \mathfrak{m}_{\psi}(\N)$ defined by $b=(b_n)_{n\in \N}\mapsto (b_{\log_2(n)}\chi_{2^\N}(n))_{n\in \N}$. Here $2^\N=\{1,2,4,8,16,\ldots\}$ denotes the dyadic natural numbers. 
It should be pointed out that the ideal $\mathfrak{m}_{\tilde{\psi}}(\N)$ supports Dixmier traces defined directly from $\tilde{\psi}$ if and only if $A_\psi(e)=1$ (in which case $\tilde{\psi}$ has regular variation of index $k_\psi=0$).

Summing up, there are constants $\alpha_0,\alpha_1>0$ such that for any exponentiation invariant extended limit $\omega$, and $c\in \mathfrak{m}_{\tilde{\psi}}^{(p)}$
\begin{equation}
\label{estofrtrTr}
\alpha_0 \mathrm{tr}_{\omega\circ \log_2,\tilde{\psi}}(|c|^p)\leq \mathrm{Tr}_{\omega,\psi}(|H_{\overline{\gamma_{p,c}}}|^p)\leq \alpha_1 \mathrm{tr}_{\omega\circ \log_2,\tilde{\psi}}(|c|^p).
\end{equation}

For a function $g\in L^\infty(0,\infty)$ define the function $\bar{g}\in L^\infty(0,\infty)$ by the formula 
$$\bar{g}(t):=\int_{\lfloor t \rfloor }^{\lfloor t \rfloor+1}g(s)\mathrm{d} s.$$ 
Set $C:=-\liminf_{t\to \infty} g(t)$ and define
\begin{equation}
\label{cnfromg}
c_n:=\left( |\bar{g}(n)+C|\cdot \tilde \psi'(n)\right)^{1/p}, \ n\ge0.
\end{equation}
It is easy to see that $c=|c|\in \mathfrak{m}_{\tilde{\psi}}^{(p)}(\N)$.

\begin{lem}
\label{integracom}
Assume that $g\in L^\infty(0,\infty)$ for some $\beta>0$ satisfies that 
\begin{equation}
\label{conditionongorfor}
g(t)-\bar{g}( t)= O(t^{-\beta}), \quad \mbox{as $t\to \infty$}.
\end{equation}
For $c=(c_n)_{n\in \N}\in\mathfrak{m}^{(p)}_{\tilde{\psi}}(\N)$ defined as in Equation \eqref{cnfromg}, it holds that 
$$\mathrm{tr}_{\omega\circ \log_2,\tilde{\psi}}(|c|^p)=\lim_{t\to\omega\circ \log_2} \frac{1}{\tilde{\psi}(t)}\int_0^t g(s)\cdot \tilde \psi'(s) \mathrm{d}s+C,$$
where $C=-\liminf_{t\to \infty} g(t)$.
\end{lem}

\begin{proof}
By the definition of liminf, we have that $g+C-|g+C|=o(1)$. It follows that the function
$$\frac{1}{\tilde{\psi}(t)}\int_0^t \sum_{j=0}^\infty|\bar{g}(j)+C|\cdot \tilde \psi'(j) \chi_{(j,j+1]}(s) \mathrm{d}s - \frac{1}{\tilde{\psi}(t)}\int_0^t \sum_{j=0}^\infty(\bar{g}(j)+C)\cdot \tilde \psi'(j) \chi_{(j,j+1]}(s)\mathrm{d}s$$
is $o(1)$ as $\to\infty$.
We therefore have
\begin{align*}
\mathrm{tr}_{\omega\circ \log_2,\tilde{\psi}}(|c|^p)&=\lim_{t\to\omega\circ \log_2} \frac{1}{\tilde{\psi}(t)}\int_0^t \sum_{j=0}^\infty|c_j|^p \chi_{(j,j+1]}(s) \mathrm{d}s\\
&=\lim_{t\to\omega\circ \log_2} \frac{1}{\tilde{\psi}(t)}\int_0^t \sum_{j=0}^\infty| \bar{g}(j)+C|\cdot \tilde \psi'(j) \chi_{(j,j+1]}(s) \mathrm{d}s\\
&=\lim_{t\to\omega\circ \log_2}\frac{1}{\tilde{\psi}(t)}\int_0^t \sum_{j=0}^\infty(\bar{g}(j)+C)\cdot \tilde \psi'(j) \chi_{(j,j+1]}(s)\mathrm{d}s.
\end{align*}

The function $\tilde{\psi}$ has regular variation, so \cite[Theorem 1.5.11]{RegVar} implies that $\frac{\tilde \psi'(t)}{\tilde \psi(t)}=o(1)$ 
as $t\to \infty$. In particular, 
$$\frac{1}{\tilde{\psi}(t)}\int_0^t \sum_{j=0}^\infty(\bar{g}(j)+C)\cdot \tilde \psi'(j) \chi_{(j,j+1]}(s) \mathrm{d}s-\frac{1}{\tilde{\psi}(t)}\int_0^t (\bar{g}(s)+C)\cdot \tilde \psi'(s) \mathrm{d}s=o(1).$$

Consider
\begin{align*}
&\left|\frac{1}{\tilde{\psi}(t)}\int_0^t (\bar{g}(s)+C)\cdot \tilde \psi'(s) \mathrm{d}s - \frac{1}{\tilde{\psi}(t)}\int_0^t (g(s)+C)\cdot \tilde \psi'(s) \mathrm{d}s\right|\\
&\leq \frac{1}{\tilde{\psi}(t)}\int_0^t |g(s)-\bar{g}(s)| \tilde \psi'(s)\mathrm{d}s. 
\end{align*}
Since $|g(t)-\bar{g}(t)|\leq \rho t^{-\beta}$ for $t\ge t_0$ for some $t_0>0$ and constant $\rho$, it follows that
\begin{align*}
&\left|\frac{1}{\tilde{\psi}(t)}\int_0^t (\bar{g}(s)+C)\cdot \tilde \psi'(s) \mathrm{d}s - \frac{1}{\tilde{\psi}(t)}\int_0^t (g(s)+C)\cdot \tilde \psi'(s) \mathrm{d}s\right|\\
&\leq \frac{2\|g\|_{L^\infty}\tilde{\psi}(t_0)}{\tilde{\psi}(t)}+\frac{\rho}{\tilde{\psi}(t)}\int_{t_0}^ts^{-\beta}\tilde \psi'(s)\mathrm{d}s.
\end{align*}
Since $\tilde{\psi}$ has regular variation of index $k_\psi$, it follows that $\tilde{\psi}'$ has regular variation of index $k_\psi - 1$ and by \cite[Theorem 1.5.11]{RegVar} we have
\begin{equation}\label{eq111}
\lim_{t\to\infty}\frac{\int_{t_0}^ts^{-\beta}\tilde \psi'(s)\mathrm{d}s}{t^{1-\beta}\tilde \psi'(t)}=\frac{1}{k_\psi-\beta}.
\end{equation}

Of course, $\beta$ can be chosen to be less than $k_\psi$. Hence,
\begin{align*}
&\left|\frac{1}{\tilde{\psi}(t)}\int_0^t (\bar{g}(s)+C)\cdot \tilde \psi'(s) \mathrm{d}s - \frac{1}{\tilde{\psi}(t)}\int_0^t (g(s)+C)\cdot \tilde \psi'(s) \mathrm{d}s\right|\\
&=o(1)+O\left( \frac{t^{1-\beta}\tilde \psi'(t)}{\tilde{\psi}(t)}\right).
\end{align*}

Since $k_\psi\neq 0$, it follows from~\eqref{eq111} that the latter estimate is, in fact, $o(1)$ and 
we conclude that condition \eqref{conditionongorfor} on $g$ ensures that
$$\frac{1}{\tilde{\psi}(t)}\int_0^t (\bar{g}(s)+C)\cdot \tilde \psi'(s) \mathrm{d}s - \frac{1}{\tilde{\psi}(t)}\int_0^t (g(s)+C)\cdot \tilde \psi'(s) \mathrm{d}s=o(1),$$
as $t\to \infty$.

Using the properties of extended limits, we obtain
\begin{align*}
\mathrm{tr}_{\omega\circ \log_2,\tilde{\psi}}(|c|^p)&=\lim_{t\to\omega\circ \log_2} \frac{1}{\tilde{\psi}(t)}\int_0^t (g(s)+C)\cdot \tilde \psi'(s) \mathrm{d}s\\
&=\lim_{t\to\omega\circ \log_2} \frac{1}{\tilde{\psi}(t)}\int_0^t g(s)\tilde \psi'(s) \mathrm{d}s+C.
\end{align*}
\end{proof}

Let $\mathrm{Lip}[0,\infty)$ denote the space of Lipschitz continuous functions on $[0,\infty)$. We define the space 
$$\mathcal{W}:=\left\{h\in L^\infty(0,\infty)\cap \mathrm{Lip}[0,\infty): \; h'(t)=O\left(\frac{1}{t}\right), \; \mbox{as $t\to \infty$}\right\}.$$ 

\begin{prop}
\label{solvingceseq}
Let $\psi: [0,\infty)\to [0,\infty)$ be a smooth increasing concave function satisfying the conditions \eqref{apsicond} and \eqref{Phi_as1} and moreover that $\psi(0)=0$, $\lim_{t\to\infty} \psi(t)=\infty$. Assume that $A_\psi(\e)\neq 1$ (see \eqref{apsicond}). Then $h\in \mathcal{W}$ if and only if $h\in L^\infty(0,\infty)$ and there exists a function $g\in L^\infty(0,\infty)$ such that 
\begin{equation}
\label{cphitildeeq}
h(t)=\frac{1}{\tilde{\psi}(t)}\int_0^t g(s)\tilde \psi'(s) \mathrm{d}s \quad\mbox{for a.e. $t$}.
\end{equation}
If $h\in \mathcal{W}$ there is a unique solution $g\in L^\infty(0,\infty)$ to Equation \eqref{cphitildeeq}.
\end{prop}

As remarked above, it poses no restriction to assume that $\psi$ is smooth by \cite[Theorem 1.8.2]{RegVar}. 

\begin{proof}
Uniqueness of solutions is clear. If $h$ solves Equation \eqref{cphitildeeq} then 
$$g(t) =\frac{(\tilde{\psi}\cdot h)'(t)}{\tilde \psi'(t)}= h(t)+ \frac{\tilde{\psi}(t)\cdot h'(t)}{\tilde \psi'(t)}$$
If $g\in L^\infty$, we conclude that  Equation \eqref{cphitildeeq}  has a solution $h\in L^\infty$ if and only if $h\in \mathrm{Lip}[0,\infty)$ and
$$h'(t)=O\left(\frac{\tilde \psi'(t)}{\tilde \psi(t)}\right).$$

Note that $A_\psi(\alpha)\neq 1$ for some $\alpha$ if and only if $A_\psi(\alpha)\neq 1$ for all $\alpha$. Moreover, $\tilde{\psi}$ has regular variation of index $k_\psi:=\log A_\psi(\e)$. By \cite[Theorem 1.5.11]{RegVar}, we have 
$$\frac{\tilde \psi'(t)}{\tilde \psi(t)}= \frac{k_\psi}{t}+o\left(\frac{1}{t}\right), \quad\mbox{as $t\to \infty$.}$$
In particular, if $k_\psi\neq 0$ then $h\in L^\infty(0,\infty)\cap \mathrm{Lip}[0,\infty)$ satisfies that $h'(t)=O\left(\frac{\tilde \psi'(t)}{\tilde \psi(t)}\right)$ if and only if $h\in \mathcal{W}$. 
\end{proof}

Let $C^{1,1}[1,\infty)$ denote the space of functions $h\in C^1[0,\infty)$ such $h'\in \mathrm{Lip}[0,\infty)$. For $\beta>0$, we define the space 
$$\mathcal{W}_\beta:=\left\{h\in\mathcal{W}\cap C^{1,1}[0,\infty): \; h''(t)=O(t^{-1-\beta}), \; \mbox{as $t\to \infty$}\right\}.$$ 

\begin{prop}
\label{gsatisfiesmean}
Let $h\in L^\infty[0,\infty)$, $\beta\in [0,1]$ and $\psi$ a function as in Proposition \ref{solvingceseq}. The equation \eqref{cphitildeeq} has a solution $g\in \mathrm{Lip}[0,\infty)$ with $g'(t)=O(t^{-\beta})$ if and only if $h\in \mathcal{W}_\beta$. In particular, if $h\in \mathcal{W}_\beta$ and $g$ solves the equation \eqref{cphitildeeq} then $g$ fulfils Condition \eqref{conditionongorfor}.
\end{prop}

\begin{proof}
We compute that 
$$g'(t)=2h'(t)+\frac{\tilde{\psi}(t)}{\tilde{\psi}'(t)}h''(t)-\frac{\tilde{\psi}(t)\tilde{\psi}''(t)}{\tilde{\psi}'(t)^2}h'(t).$$
Since \eqref{cphitildeeq} has a solution, $h'(t)=O(t^{-1})$ by Proposition \ref{solvingceseq}. Moreover, by the same argument as in Proposition \ref{solvingceseq}, $\frac{\tilde{\psi}(t)}{\tilde{\psi}'(t)}=O(t)$ whenever $k_\psi\neq 0$ and since $\tilde{\psi}'$ has regular variation, we can also conclude that $\frac{\tilde{\psi}''(t)}{\tilde{\psi}'(t)}=O(t^{-1})$. In particular, 
$$g'(t)=\frac{\tilde{\psi}(t)}{\tilde{\psi}'(t)}h''(t)+O(t^{-1})=O(t)h''(t)+O(t^{-1}).$$
It follows that $g'(t)=O(t^{-\beta})$ if and only if $h\in \mathcal{W}_\beta$.

Finally, the mean value theorem for integrals guarantees that for some $\xi\in [\lfloor t\rfloor, \lfloor t\rfloor+1]$, $\bar{g}(t)=g(\xi)$. The mean value theorem for derivatives guarantees that if $g$ satisfies $g'(t)=O(t^{-\beta})$ then
$$|\bar{g}(t)-g(t)|\leq \sup_{s\in [\lfloor t\rfloor,\lfloor t\rfloor+1]}|g'(s)|=O(t^{-\beta}).$$
\end{proof}

For $b> 1$, we write $\exp_b(x):=b^x$ with the obvious notation $\exp=\exp_\e$. For any translation invariant extended limit $\eta$ on $L^\infty$ we define the extended limit $\omega$ by
$$\omega(f):=\eta(f\circ \exp \circ \exp_2), \ f \in L^\infty.$$

Recall the notation $(P_\alpha f)(t)=f(t^\alpha)$ for $\alpha\ge1$. We also consider the operator $T_l : L^\infty \to L^\infty$, $(T_lf)(t)=f(t+l)$ defined for $l>0$. We say that $\eta$ is translation invariant if $\eta\circ T_l=\eta$ for all $l>0$.
For every $\alpha\geq 1$ we have
\begin{align}
\label{expinvariance}
\omega(P_\alpha f)&=\eta((P_\alpha f)\circ \exp \circ \exp_2)=\eta(\sigma^\alpha(f\circ \exp) \circ \exp_2)\\
\nonumber
&=\eta(T_{\log_2 \alpha}(f\circ \exp \circ \exp)).
\end{align}
Hence, $\omega$ is exponentiation invariant if and only if $\eta$ is translation invariant.

We define the space
$$\mathcal{E}:=\{h\in L^\infty(0,\infty) \ : \ h(t+l)-h(t)=o(1), \ t\to\infty, \; \forall l>0\}.$$
The reader should note that we have the inclusion $\mathcal{W}\subseteq \mathcal{E}$. Moreover, $E\subseteq L^\infty(0,\infty)$ is by definition a closed subspace.

\begin{prop}
\label{proponlimif}
For any $h\in \mathcal{E}$ there are translation invariant extended limits $\eta_1$ and $\eta_2$ such that 
$$\lim_{t\to\eta_1} h(t) = \liminf_{t\to \infty} h(t), \quad\mbox{and}\quad \lim_{t\to\eta_2} h(t) = \limsup_{t\to \infty} h(t).$$
\end{prop}

\begin{proof}
By the Hahn-Banach theorem we can find singular states $\eta_1',\eta_2'\in \mathcal{E}^*$ such that 
$$\eta_1'(h) = \liminf_{t\to \infty} h(t), \quad\mbox{and}\quad \eta_2'(h) = \limsup_{t\to \infty} h(t).$$
The action by translations preserves $\mathcal{E}$ and acts trivially modulo the closure of the compactly supported elements of $L^\infty(0,\infty)$. Therefore, the invariant Hahn-Banach theorem (see e.g. \cite[Theorem 3.3.1]{Edwards}) implies that $\eta_1',\eta_2'\in \mathcal{E}^*$ extend to translation invariant extended limits $\eta_1,\eta_2\in L^\infty(0,\infty)^*$ with $\eta_1'=\eta_1|_{\mathcal{E}}$ and $\eta_2'=\eta_2|_{\mathcal{E}}$. The proposition follows.
\end{proof}

Let us summarize the outcome of the above results on Dixmier traces. 

\begin{prop}
\label{smalldixcomp}
Let $p\in [1,\infty)$ and let $\psi: [0,\infty)\to [0,\infty)$ be an increasing concave function satisfying the conditions \eqref{apsicond} and \eqref{Phi_as1} and moreover that
$\psi(0)=0$, $\lim_{t\to\infty} \psi(t)=\infty$. Assume that $k_\psi\neq 0$. 

Let $\beta\in (0,1]$ and assume that $h\in \mathcal{W}_\beta$ satisfies that $h\circ \exp\in \mathcal{E}$ and take $c=(c_n)_{n\in\N}$ given as in \eqref{cnfromg} by
$$c_n:=(|\bar{g}(n)+C|\cdot \tilde{\psi}'(n))^{1/p},$$
where $g$ solves \eqref{cphitildeeq} and $C:=-\liminf_{t\to \infty} g(t)$. Then there are exponentiation invariant extended limits $\omega_1,\omega_2\in L^\infty(0,\infty)^*$ such that 
$$\mathrm{tr}_{\omega_j\circ \log_2,\tilde{\psi}}(|c|^p)=\begin{cases}
\liminf_{t\to \infty} h(t)-\liminf_{t\to \infty} g(t), \quad &j=1,\\
&\\
\limsup_{t\to \infty} h(t)-\liminf_{t\to \infty} g(t), \quad &j=2.\end{cases}$$

\end{prop}

\begin{proof}
We note that $g$ exists by Proposition \ref{solvingceseq}. Since $h\in \mathcal{W}_\beta$ for a $\beta>0$, $g$ satisfies that $\bar{g}(t)-g(t)=O(t^{-\beta})$ by Proposition \ref{gsatisfiesmean}. By Lemma \ref{integracom}, for any exponentiation invariant extended limit $\omega$, 
\begin{equation}
\label{transdmdmda}
\mathrm{tr}_{\omega\circ \log_2,\tilde{\psi}}(|c|^p)=\lim_{t\to\omega\circ \log_2} h(t)+C.
\end{equation}
Since $h\circ \exp\in \mathcal{E}$ we can take translation invariant extended limits $\eta_1$ and $\eta_2$ as in Proposition \ref{proponlimif} such that 
\begin{align}
\label{transdmdmdatwo}
\lim_{t\to\eta_1\circ \exp} h(t) &= \liminf_{t\to \infty} h(\e^t)=\liminf_{t\to \infty} h(t), \quad\mbox{and}\\
\nonumber
\lim_{t\to\eta_2\circ \exp} h(t) &= \limsup_{t\to \infty} h(\e^t)=\limsup_{t\to \infty} h(t).
\end{align}
Define the extended limits $\omega_j:=\eta_j\circ \exp\circ\exp_2$, $j=1,2$, which are exponentiation invariant because $\eta_1$ and $\eta_2$ are translation invariant. We conclude the proposition from combining the two statements \eqref{transdmdmda} and \eqref{transdmdmdatwo} with the fact that $\omega_j\circ \log_2=\eta_j\circ \exp$.
\end{proof}

\begin{lem}
\label{cnonstructionofh}
Let $\psi: [0,\infty)\to [0,\infty)$ be an increasing concave function satisfying the conditions \eqref{apsicond} and \eqref{Phi_as1} and such that
$\psi(0)=0$, $\lim_{t\to\infty} \psi(t)=\infty$, $k_\psi\neq 0$.
For every $h_0\in C^{1,1}[0,\infty)$ such that $h_0, h_0', h_0''\in L^\infty(0,\infty)$, the function 
$$h(t):=h_0(\log(1+\log (1+t))), \ t>0$$
belongs to $\mathcal{W}_1$ and satisfy the following:
\begin{itemize}
\item $h\circ \exp\in \mathcal{E}$,
\item $\liminf_{t\to\infty} h(t)=\liminf_{t\to \infty} g(t)$. 
\end{itemize}
Here $g$ denotes the solution to \eqref{cphitildeeq}. 

Moreover, the function $h_0$ is not convergent as $t\to\infty$
if and only if 
$$\limsup_{t\to\infty} h(t)>\liminf_{t\to \infty} g(t).$$
\end{lem}

\begin{proof}
Since 

$$h'(t)=h_0'(\log(1+\log (1+t)))\frac{1}{\log (1+t)}\frac{1}{1+t},$$
we have that $h\in \mathcal{W}$. Moreover, 
\begin{align*}
h''(t)&=h_0''(\log(1+\log (1+t)))\frac{1}{(\log (1+t))^2}\frac{1}{(1+t)^2}\\
&-h_0'(\log(1+\log (1+t))\frac{1}{(1+t)^2}\left(\frac{1}{(\log (1+t))^2}+\frac{1}{\log (1+t)}\right)=O(t^{-2}),
\end{align*}
so $h\in \mathcal{W}_1$.

Since $$(h\circ \exp)'(t)=h_0'(\log(1+\log (1+\e^t)))\frac{1}{1+\log (1+\e^t)}\frac{\e^t}{1+\e^t}=O(\frac{1}{t}),$$ it follows that $h\circ \exp\in \mathcal{W}\subseteq \mathcal{E}$. Solving equation \eqref{cphitildeeq} for $g$, we obtain
\begin{align*}
g(t) &=\frac{(\tilde{\psi}\cdot h)'(t)}{\tilde \psi'(t)}=h(t)+\frac{\tilde{\psi}(t)}{\tilde \psi'(t)}h'(t).
\end{align*}
Since $k_\psi\neq 0$, we have $\frac{\tilde{\psi}(t)}{\tilde \psi'(t)}=O(t)$. Thus, the fact that $h'(t)=o(t^{-1})$ implies that $g(t)=h(t)+o(1)$. Therefore 
$$\liminf_{t\to\infty} g(t)=\liminf_{t\to\infty} h(t)=\liminf_{t\to \infty} h_0(t).$$ 
It is clear that
$$\limsup_{t\to\infty} h(t)=\limsup_{t\to\infty} h_0(t).$$
We can conclude that the function $h_0$ is not convergent as $t\to\infty$ if and only if $\limsup_{t\to\infty} h(t)>\liminf_{t\to \infty} g(t)$.
\end{proof}

\begin{thm}
\label{nonmeasthm}
Let $p\in [1,\infty)$, $\psi: [0,\infty)\to [0,\infty)$ be as in Proposition \ref{smalldixcomp} and assume that $h_0\in C^{1,1}[0,\infty)$ is such that $h_0, h_0', h_0''\in L^\infty(0,\infty)$ and the function $h_0$ is not convergent as $t\to\infty$. 
Define the holomorphic function 
$$f(z):=\sum_{n=0}^\infty 2^{-n/p}c_n z^{2^n},$$
where $c=(c_n)_{n\in\N}$ is given as in \eqref{cnfromg} from the solution $g$ to \eqref{cphitildeeq} for $h(t):=h_0(\log(1+\log(1+t)))$. Then $f\in B_{p,\psi}(S^1)$ and moreover $|H_f|^p\in \mathcal{M}_\psi$ is non-measurable. More precisely, $f$ satisfies that there are exponentiation invariant extended limits $\omega_1$ and $\omega_2$ such that 
$$\mathrm{Tr}_{\omega_1,\psi}(|H_{\overline{f}}|^p)=0\quad \mbox{and}\quad \mathrm{Tr}_{\omega_2,\psi}(|H_{\overline{f}}|^p)>0.$$
\end{thm}

\begin{proof}
By Lemma \ref{cnonstructionofh} $h\in \mathcal{W}_1$ satisfies that $h\circ \exp\in \mathcal{E}$  and $g$ satisfies that $\liminf_{t\to\infty} h(t)=\liminf_{t\to \infty} g(t)$ and $\limsup_{t\to\infty} h(t)>\liminf_{t\to \infty} g(t)$. It follows from Proposition \ref{smalldixcomp} that there are exponentiation invariant extended limits $\omega_1,\omega_2\in L^\infty(0,\infty)^*$ such that
$$\mathrm{tr}_{\omega_1\circ \log_2,\tilde{\psi}}(|c|^p)=0\quad\mbox{and}\quad \mathrm{tr}_{\omega_2\circ \log_2,\tilde{\psi}}(|c|^p)>0.$$

By positivity of the Dixmier trace and \eqref{estofrtrTr}, we have that 
\begin{align*}
0\leq &\mathrm{Tr}_{\omega_1,\psi}(|H_{\overline{f}}|^p)\leq \alpha_1\mathrm{tr}_{\omega_1\circ \log_2,\tilde{\psi}}(|c|^p)=0\quad\mbox{and}\\
&\mathrm{Tr}_{\omega_2,\psi}(|H_{\overline{f}}|^p)\geq \alpha_0\mathrm{tr}_{\omega_2\circ \log_2,\tilde{\psi}}(|c|^p)>0.
\end{align*}
\end{proof}

\begin{rem}
Theorem \ref{nonmeasthm} extends \cite[Theorem 4]{EZ} to general $p$ and general $\psi$ with $k_\psi\neq 0$. Our proof is longer. The length is not just due to the reason that we are in a more general setting. The reason for the length of the proof is two-fold. Firstly, we wanted to better understand the mechanism that creates non-measurability in terms of functions $h$ as in Proposition \ref{smalldixcomp}. Secondly, we wanted to improve the construction of the two exponentiation invariant extended limits $\omega_0$ and $\omega_1$ that realizes the non-measurability as is done in Proposition \ref{smalldixcomp}. 

The construction in the proof of \cite[Theorem 4]{EZ} starts from an extended limit $\eta\in \ell^\infty(\N)^*$ and is used to construct two different extended limits $\omega_1$ and $\omega_2$ on $L^\infty(0,\infty)$. The process of going from sequences to function is delicate when it comes to extended limits. In \cite{EZ}, starting from a translation invariant extended limit $\eta\in \ell^\infty(\N)^*$ and the mapping $b_j:\N\to \R_+$, $b_j(n):=a^{(2k+j)\pi}$ for an $a>1$ and $j=1,2$, Engli\v{s}-Zhang \cite{EZ} defined extended limits
$$\omega_j(f):=\eta(((M(f\circ \exp))\circ b_j), \quad \mbox{for $f\in L^\infty(0,\infty)$}.$$
Here $M$ denotes the logarithmic Cesaro mean. Since $\eta$ is only invariant for translations by natural numbers, a computation as in Equation \eqref{expinvariance} shows that $\omega_j$ need only satisfy $\omega_j\circ P_\alpha=\omega_j$ for $\alpha$ in the multiplicative subgroup of $\R_+$ generated by $a^2$. To our knowledge, one needs full exponentiation invariance in order for a relation as in Theorem \ref{allthethingsGSsaid} part i) to hold. It is unclear to us how the conclusion \cite[Theorem 4]{EZ} is reached from only knowing invariance with respect to $P_{a^2}$. Proposition \ref{smalldixcomp} above circumvents this problem. 
\end{rem}

\end{document}